\documentclass[12pt,english]{article}

\usepackage{amssymb}
\usepackage{polski}
\usepackage{ifthen}
\selecthyphenation{polish}
\mathchardef\ogon="012C%
\newcommand{\as}{a\kern-0.22em\lower.40ex\hbox{$_{\ogon}$}}
\newcommand{\As}{A\kern-0.22em\lower.40ex\hbox{$_{\ogon}$}}
\newcommand{\es}{e\kern-0.24em\lower.40ex\hbox{$_{\ogon}$}}
\newcommand{\Es}{E\kern-0.22em\lower.40ex\hbox{$_{\ogon}$}}
\makeatletter

\usepackage[a4paper, left=2.5cm, right=2.5cm, top=2.5cm, bottom=3.5cm, headsep=1.2cm]{geometry}
\newtheorem{theorem}{Theorem}[section]

\newtheorem{corollary}[theorem]{Corollary}

\newtheorem{definition}[theorem]{Definition}
\newtheorem{example}[theorem]{Example}

\newtheorem{proposition}[theorem]{Proposition}
\newtheorem{remark}[theorem]{Remark}

\newenvironment{proof}[1][Proof]{\noindent\textbf{#1.} }{\ \rule{0.5em}{0.5em}}
\def\qed{\hbox to 0pt{}\hfill$\rlap{$\sqcap$}\sqcup$}
\makeatother
\usepackage{amssymb}
\usepackage{amsmath}
\usepackage{amsfonts}
\usepackage{color}
\usepackage{xcolor}
\usepackage{graphicx}
\usepackage{graphics}
\numberwithin{equation}{section}
\usepackage{babel}
\date{}
\title{On minimax theorems for lower semicontinuous functions in Hilbert spaces }
\author{Ewa M. Bednarczuk
\thanks{System Research Institute, Polish Academy of Sciences, Newelska 6, 01--447, Warsaw, Poland,  Ewa.Bednarczuk@ibspan.waw.pl
} \thanks{Warsaw University of Technology, Faculty of Mathematics and Information Science, ul. Koszykowa 75,
00--662 Warsaw, Poland, E.Bednarczuk@mini.pw.edu.pl}
, Monika Syga
\thanks{Warsaw University of Technology, Faculty of Mathematics and Information Science, ul. Koszykowa 75,
 00--662 Warsaw, Poland,  M.Syga@mini.pw.edu.pl} 
 }
\begin{document}
\maketitle

\begin{abstract}
We prove minimax theorems for lower semicontinuous functions defined on a Hilbert space.
 The main tool is the theory of $\Phi$-convex functions and 
sufficient and necessary conditions for the minimax equality  for  $\Phi$-convex  functions.
These conditions are expressed in terms of abstract  $\Phi$-subgradients.

\medskip{}

 \medskip{}

\end{abstract}
\section{Introduction}
Let $X$  and  $Y$ be nonempty sets and let $a:X\times Y\rightarrow\hat{\mathbb{R}}:=\mathbb{R}\cup\{\pm\infty\}$,  be
an extended  real-valued function defined on $X\times Y$. Minimax theorems provide conditions for the minimax equality
$$
\sup_{y\in Y} \inf_{x\in X} a(x,y)=\inf_{x\in X} \sup_{y\in Y} a(x,y).
$$
to hold for $a$. An exhaustive survey of minimax theorems is given e.g. in \cite{simons}.

In this paper we prove minimax theorems for lower semicontinuous functions 
defined on a Hilbert space and bounded
from below by a quadratic function.
The novelty of the results presented is  that we do not exploit any compactness and/or connectedness assumptions.
Our approach is based on minimax theorems for abstract convex functions 
obtained in \cite{bed-syg, phdsyga, syga, sygaconv}.

We start with definitions related to abstract convexity. 
Let $X$ be a set. Let $\Phi$ be a set of real-valued functions $\varphi:X\rightarrow \mathbb{R}$. 

For any $f,g:X\rightarrow\hat{\mathbb{R}}$ 
$$
f\le g\ \Leftrightarrow\  f(x)\le g(x)\ \ \forall x\in X.
$$
Let $f:X\rightarrow\hat{\mathbb{R}}$. 
The set
$$
\text{supp}(f,\Phi):=\{\varphi\in \Phi\ :\ \varphi\le f\}
$$
is called the {\em support} of $f$ with  respect to $\Phi$.
We will use the notation $\text{supp}(f)$ 
if the class $\Phi$ is clear from the context. 
\begin{definition}(\cite{dolecki-k, rolewicz, rubbook})
	\label{convf}
	A function $f:X\rightarrow
	\hat{\mathbb{R}}$ is called {\em $\Phi$-convex} if
	$$
	f(x)=\sup\{\varphi(x)\ :\ \varphi\in\textnormal{supp}(f)\}\ \ \forall\ x\in X.
	$$
\end{definition}
 By convention, $\textnormal{supp}(f)=\emptyset$ if and only if $f\equiv-\infty$.
	In this paper we always assume that $\textnormal{supp}(f)\neq\emptyset$, i.e. we limit our attention to functions $f:X\rightarrow\bar{\mathbb{R}}:=\mathbb{R}\cup\{+\infty\}$. 
	
	We say that a function $f:X\rightarrow\bar{\mathbb{R}}$ is proper if the effective domain of $f$ is nonempty, i.e.
	$$
	\text{dom}(f):=\{x\in X \ : \ f(x)<+\infty \}\neq \emptyset.
	$$


If, for every $y\in Y$  the function $a(\cdot,y)$ is $\Phi$-convex then, a sufficient and  necessary
condition for $a$ to satisfy the minimax equality is the so called intersection property
introduced in \cite{bed-syg} and investigated in \cite{phdsyga, syga, sygaconv}. Let $\varphi_1,\varphi_2:X\rightarrow \mathbb{R}$ be any functions from the set $\Phi_{lsc}$ defined below and $\alpha\in \mathbb{R}$. We say that 
{\em the intersection property} holds for $\varphi_1$ and $\varphi_2$ on $X$ at the level $\alpha$ if and only if
$$
[\varphi_1<\alpha]\cap[\varphi_2<\alpha]=\emptyset,
$$
where $[\varphi<\alpha]:=\{x\in X\ :\ \varphi(x)<\alpha\}$ is the strict lower level set 
of function $\varphi:X\rightarrow \mathbb{R}$.

From now on we consider classes $\Phi$ which are closed under vertical shift, i.e. $\varphi+c\in\Phi$ for any $\varphi\in\Phi$ and $c\in\mathbb{R}$. 
The following result is proved in \cite{phdsyga, sygaconv}.
\begin{theorem}(Theorem 1.2 of \cite{sygaconv}, see also Theorem 3.3.3 of \cite{phdsyga}).
	\label{min-max}
	Let $X$ be a nonempty set and $Y$ be a real vector space and let $a:X\times Y\rightarrow\bar{\mathbb{R}}$. Assume that for any $y\in Y$ the  function $a(\cdot,y):X\rightarrow\bar{\mathbb{R}}$ 
	is proper $\Phi$-convex on $X$ and for any $x\in X$ the function $a(x,\cdot):Y\rightarrow\bar{\mathbb{R}}$  is concave on $Y$. The following conditions are equivalent:
	\begin{description}
		\item [{\em (i)}] for every $\alpha\in\mathbb{R}$, $\alpha < \inf\limits_{x\in X} \sup\limits_{y\in Y} a(x,y)$, there exist $y_{1}, y_{2}\in Y$ and $\varphi_{1}\in \textnormal{supp } a(\cdot, y_{1})$, $\varphi_{2}\in \textnormal{supp } a(\cdot, y_{2})$ such that the intersection property holds for $\varphi_{1}$ and $\varphi_{2}$ on $X$ at the level $\alpha$,
		\item [{\em (ii)}] $\sup\limits_{y\in Y} \inf\limits_{x\in X} a(x,y)=\inf\limits_{x\in X} \sup\limits_{y\in Y} a(x,y).$
	\end{description}
\end{theorem}
Condition $(i)$ of Theorem \ref{min-max} requires  the existence of two functions from support sets $\textnormal{supp } a(\cdot, y_{1})$ and $\textnormal{supp } a(\cdot, y_{2})$ for which the intersection property holds on $X$ at the level $\alpha < \inf\limits_{x\in X} \sup\limits_{y\in Y} a(x,y)$. A natural question is whether it is possible to choose these two functions as $\Phi$-subgradients of $a(\cdot, y_{1})$ and $a(\cdot, y_{2})$ respectively, at some points. 

In the present paper we investigate the possibility of expressing condition $(i)$ of Theorem \ref{min-max} with the help of $\Phi$-subgradients. We consider $\Phi$-convex functions $a(\cdot, y)$ for two classes $\Phi$. The first class is defined as
$$
\Phi_{conv}:=\{\varphi : X \rightarrow \mathbb{R},\ \varphi(x)= \left\langle \ell,x\right\rangle+c, \ \ x\in X,\  \ell\in X^{*}, \ c\in \mathbb{R}\}, 
$$
where $X$ is a vector topological space, $X^{*}$ is a topological dual space to $X$. It is a well known fact
(see for example Proposition 3.1 of \cite{ek-tem}) that a proper convex lower semicontinuous function $f:X\rightarrow \bar{\mathbb{R}}$ is  $\Phi_{conv}$-convex.

The second class which is considered in this paper is defined in the following way
$$
\Phi_{lsc}:= \{\varphi : X \rightarrow \mathbb{R}, \ \varphi(x)=-a\|x\|^2+ \left\langle \ell,x\right\rangle+c, \ \ x\in X,\  \ell\in X^{*}, \ a\geq 0, \ c\in \mathbb{R} \},
$$ 
where $X$ is a normed space.

 In the following theorem we provide a further characterization of $\Phi_{lsc}$-convex functions.

\begin{theorem}(\cite{rubbook}, Example 6.2)
	\label{rub-lsc-1}
	Let $X$ be a Hilbert space.
	Let $f: X\rightarrow  \bar{\mathbb{R}}$ be lower semicontinuous on $X$. If there exists $\bar{\varphi}\in \Phi_{lsc}$ such that $\bar{\varphi}<f$, then $f$ is $\Phi_{lsc}-$convex. 
\end{theorem}
 Let us note that when $X=\mathbb{R}^{n}$ the class of $\Phi_{lsc}$-convex  functions contains 
 the class of prox-bounded functions $f$ (Definition 1.23 in \cite{west-rock}),  which  possess proximal subgradients (Definition 8.45 of \cite{west-rock}) at each $x\in \mathbb{R}^{n}$ (if $\text{dom}f=\mathbb{R}^{n}$)
 (Proposition 8.46f \cite{west-rock}). See also \cite{ber-thi}.

The organisation of the paper is as follows. In section 2 we collect basic facts on $\Phi$-subdifferentials of $\Phi$-convex functions. In section 3 we provide characterizations of the intersection property in terms of $\Phi$-$\varepsilon$-subdifferentials. In section 4 we discuss our main tool which is the Br\o{}nsted-Rockafellar-type theorem for $\Phi_{lsc}$-convex functions. In section 5 we discuss
the intersection property in the class $\Phi_{conv}$. Section 6 contains  our main results.

\section{$\Phi$-subdifferentials }
Let $X$ be a set. The $\Phi$-subdifferential of a  function $f:X\rightarrow \bar{\mathbb{R}}$ at a point $\bar{x}\in \text{dom}(f)$ is defined as follows
$$
\partial_{\Phi}f(\bar{x}):=\{\varphi\in \Phi \ :\ \varphi(x)-\varphi(\bar{x})\leq f(x)-f(\bar{x}), \ \forall \ x\in X \}.
$$
The elements $\varphi$ of $\partial_{\Phi}f(\bar{x})$ are called $\Phi$-subgradients of $f$ at $\bar{x}$. This definition and some results related to $\Phi$-subdifferentials can be found e.g. in \cite{rolewicz} and \cite{sharikov}.

Let $\varepsilon >0$. The $\Phi$-$\varepsilon$-subdifferential  of $f$ at the point $\bar{x}\in\text{dom}(f)$ is defined as follows
$$
\partial_{\Phi}^{\varepsilon}f(\bar{x}):=\{\varphi\in \Phi \ :\ \varphi(x)-\varphi(\bar{x})-\varepsilon\leq f(x)-f(\bar{x}), \ \forall \ x\in X \}.
$$
Whenever the class $\Phi$ is clear from the context we use the notation $\varepsilon$-subdifferential. The $\Phi$-$\varepsilon$-subdifferential  was defined in \cite{ioffe-rub}, which is a direct adaptation of the definition from \cite{bro-rock}. The elements $\varphi$ of $\partial^{\varepsilon}_{\Phi}f(\bar{x})$ are called 
$\Phi$-$\varepsilon$-subgradients ($\varepsilon$-subgradients) of $f$ at $\bar{x}$.

The following proposition states a necessary and sufficient condition for the nonemptiness of $\Phi$-subdifferential and $\Phi$-$\varepsilon$-subdifferential of  $f$ at a given point $\bar{x}$. Similar results can be found in \cite{rubbook} (Proposition 1.2,  Corollary 1.2) and \cite{rubbook} (Proposition 5.1).
\begin{proposition}
	\label{prop_subdiff}
	Let $f:X\rightarrow\bar{\mathbb{R}}$ be a proper $\Phi$-convex functions, $\varepsilon\geq 0$ and $\bar{x}\in \textnormal{dom}(f)$. Then
	\begin{equation}
	\label{eq-subdiff}
	\partial^{\varepsilon}_{\Phi}f(\bar{x})=\{\varphi\in\textnormal{supp}(f)\ :\ f(\bar{x})=\varphi(\bar{x})+\varepsilon\}
	\end{equation}
\end{proposition}
\begin{proof}
	 Let $\varphi\in \partial^{\varepsilon}_{\Phi}f(\bar{x})$. By the definition,
	$$
	f(x)\geq \varphi(x)-\varphi(\bar{x}) + f(\bar{x})- \varepsilon.
	$$
	Let $\bar{\varphi}:= \varphi-\varphi(\bar{x}) + f(\bar{x}) - \varepsilon$, then  $\bar{\varphi}\in\text{supp}(f)$ and $\bar{\varphi}(\bar{x}) + \varepsilon=f(\bar{x})$.
	
	 Let $\bar{\varphi}\in\text{supp}(f) $ be such that $f(\bar{x})=\bar{\varphi}(\bar{x})+\varepsilon$. By the fact that $\bar{\varphi}\in\text{supp}(f) $ we have $f(x)\geq\bar{\varphi}(x) $, for all $x\in X$. Hence
	the following inequality holds
	$$
	\bar{\varphi}(x)-\bar{\varphi}(\bar{x})-  \varepsilon\leq f(x)-f(\bar{x}), \ \forall \ x\in X,
	$$
	which means $\bar{\varphi}\in \partial^{\varepsilon}_{\Phi}f(\bar{x})$.
\end{proof}

\section{The intersection property via $\Phi$-$\varepsilon$-subgradients }

 Since $[\bar{\varphi}<\alpha]\subset [\varphi<\alpha] $ whenever $\varphi\leq\bar{\varphi}$, we have the following proposition.
 \begin{proposition}(\cite{phdsyga}, Proposition 2.1.3)
 	\label{prop-przeciecie}
 	Let $\alpha\in \mathbb{R}$ and $\varphi_{1},\varphi_{2}, \bar{\varphi_{1}}, \bar{\varphi_{2}}:X\rightarrow \mathbb{R}$ be such that 
 	$$
 	\bar{\varphi_{1}}\geq \varphi_{1} \ \ \text{and} \ \ \ 
 	\bar{\varphi_{2}}\geq \varphi_{2}.
 	$$
 	If $\varphi_{1}$ and $\varphi_{2}$ have the  intersection property on $X$ at the level $\alpha$, then $\bar{\varphi_{1}}$ and $\bar{\varphi_{2}}$ have the intersection property on $X$ at the level $\alpha$.
 \end{proposition}
 
	For $\varepsilon$-subdifferentials the following theorem holds.

\begin{theorem}
	\label{th-epsilon}
	Let $X$ be a set.
	Let $\Phi$ be a class of functions $\varphi:X\rightarrow\mathbb{R}$.
	Let $f,g:X\rightarrow \bar{\mathbb{R}}$ be proper $\Phi$-convex functions and $\alpha\in\mathbb{R}$.
	
	The following conditions are equivalent:
	\begin{itemize}
	 \item[{\em (i)}]  there exist $\varphi_1\in\textnormal{supp}(f)$ and $\varphi_2\in\textnormal{supp}(g)$ such that $\varphi_1,\varphi_2$ have the  intersection property on $X$ at the level $\alpha$,
	 
	\item[{\em (ii)}] for any $\varepsilon>0$ there exist $x_{1} \in \textnormal{dom}(f)$, $x_{2}\in \textnormal{dom}(g)$ and $\bar{\varphi_1}\in\partial_{\Phi}^{\varepsilon}f(x_{1})$ and $\bar{\varphi_2}\in\partial_{\Phi}^{\varepsilon}g(x_{2})$ such that $\bar{\varphi_{1}},\bar{\varphi_{2}}$ have the intersection property on $X$ at the level $\alpha$.
	\end{itemize}
\end{theorem}
\begin{proof} $(ii)\Rightarrow(i)$. Follows 
	immediately from Proposition \ref{prop_subdiff}.
	
	\noindent
	 $(i)\Rightarrow(ii)$. Since  $\varphi_1\in\textnormal{supp}(f)$ and $\varphi_2\in\textnormal{supp}(g)$, we have
	$$
	\inf\limits_{x\in X}\{ f(x)-\varphi_{1}(x)\}=:c_1\ge 0 \ \ \text{and} \ \ \inf\limits_{x\in X} \{g(x)-\varphi_{2}(x)\}=:c_2\ge 0.
	$$
	Let $\varepsilon>0$. There exist  $x_{1} \in \textnormal{dom}(f)$ and $x_{2}\in \textnormal{dom}(g)$
	such that
	$$
	f(x_{1})-\varphi_{1}(x_{1})<c_{1}+\varepsilon \ \ \text{and} \ \ g(x_{2})-\varphi_{2}(x_{2})<c_{2}+\varepsilon.
	$$
	Define
	$$
	\bar{\varphi_{1}}(x):= \varphi_{1}(x)+c_1 \ \ \text{and} \ \ \bar{\varphi_{2}}(x):=\varphi_{2}(x)+ c_2.$$ 
	Functions $\bar{\varphi_{1}}, \bar{\varphi_{2}}$ belong to $\Phi$ and
		$$
		-f(x_{1})>-\bar{\varphi_{1}}(x_{1})-\varepsilon \ \ \text{and} \ \ -g(x_{2})>-\bar{\varphi_{2}}(x_{2})-\varepsilon.
		$$
	Notice that
	$$
	\bar{\varphi_{1}}(x)\leq f(x) \ \text{and} \ \ \bar{\varphi_{2}}(x)\leq g(x) \ \ \text{for all} \ \ x\in X.
	$$
	So,
	$\bar{\varphi_1}\in\partial_{\Phi}^{\varepsilon}f(x_{1})$ and $\bar{\varphi_2}\in\partial_{\Phi}^{\varepsilon}g(x_{2})$.
	It is obvious that $\bar{\varphi_{1}}\geq \varphi_{1}$ and  $\bar{\varphi_{2}}\geq \varphi_{2}$ and by Proposition \ref{prop-przeciecie},   $\bar{\varphi_{1}}$ and $\bar{\varphi_{2}}$ have the  intersection property on $X$ at the level $\alpha$.
\end{proof}

A characterization of the intersection property in terms
of $\Phi$-subgradients requires more advanced tools as well as  additional knowledge about properties of functions $\varphi$ from a given class $\Phi$.
We provide characterizations of the intersection property in terms of $\Phi$-subgradients for the class $\Phi_{lsc}$ (section 4) and for the class $\Phi_{conv}$ (section 5).

\section{The intersection property for $\Phi_{lsc}$-subgradients}

In this section we express  the intersection property for $\Phi_{lsc}$-convex functions  in terms of $\Phi_{lsc}$-subgradients. The main result of the section is Theorem \ref{th-subdif}.

To this aim we use  the Borwein-Preiss variational principle in the following form.

\begin{theorem}(\cite{schirotzek}, Theorem 8.3.3)
	\label{bp}
	Assume that $X$ is a Hilbert space and the function $f:X\rightarrow \bar{\mathbb{R}}$ is proper lower semicontinuous on $X$. Let $\varepsilon>0$ and let $y\in X$	be such that $f(y)\leq \inf\limits_{x\in X}f(x)+\varepsilon$. Then for any $\lambda>0$ there exist $z\in X$, such that
$$
\|z-x\|\leq \lambda, \ \ \ \ \ f(z)\leq f(x)+ \frac{\varepsilon}{\lambda^2}\|x-z\|^2 \ \ \ \ \text{for all} \ \ \ x\in X.
$$
\end{theorem}

With the help of this variational principle we prove the following variant of the Br\o{}nsted-Rockafellar theorem for $\Phi_{lsc}$-convex functions. The proof we present below is an adaptation of  the  proof of   Proposition 5.3 of \cite{ioffe-rub}  to our definition of subgradients.

\begin{theorem}
	\label{i-r}
	Let $\varepsilon>0$, $f:X\rightarrow\bar{\mathbb{R}}$ be a proper $\Phi_{lsc}-$convex function, $y\in \textnormal{dom}(f)$ and \\ $\varphi(\cdot)=-a\|\cdot\|^2+\left\langle \ell,\cdot\right\rangle+c \in \partial_{\Phi_{lsc}}^{\varepsilon}f(y) $.  For every $\lambda >0$ there exist $\bar{y}\in  \textnormal{dom}(f)$ and \\ $\bar{\varphi}(\cdot)=-\bar{a}\|\cdot\|^2+\left\langle \bar{\ell},\cdot\right\rangle+\bar{c} \in\partial_{\Phi_{lsc}}f(\bar{y}) $   such that
	$$
	\|y- \bar{y}\|\leq \lambda, \ \ \  \|\ell-\bar{\ell}\| \leq\frac{2\varepsilon}{\lambda^2}(\lambda + \|y\|), \  \ \bar{a}-a= \frac{\varepsilon}{\lambda^2} \ \  \text{and} \ \ c-\bar{c}\leq\frac{\varepsilon}{\lambda^2}\|\bar{y}\|^2.
	$$
\end{theorem}
\begin{proof}
	Since $\varphi\in \partial_{\Phi_{lsc}}^{\varepsilon}f(y) $ we have, by Proposition \ref{prop_subdiff}
	  $$
	  f(x)-f(y)\geq 
	  -a\|x\|^2 +\left\langle \ell,x\right\rangle +a\|y\|^2 -\left\langle \ell,y\right\rangle -  \varepsilon \ \ \ \text{for all} \ \ \ x\in X.
	  $$
	 which can be rewritten as
	$$
	W(y)\leq\inf_{x\in X}W(x)+\varepsilon,
	$$
	 where $W(x):=f(x)+a\|x\|^2-  \left\langle \ell,x\right\rangle$.
	By Theorem \ref{bp}, for every $\lambda>0$ there exists $\bar{y}\in X$ such that $\|y- \bar{y}\|\leq \lambda$ and 
	$$
	W(\bar{y})\leq W(x)+ \frac{\varepsilon}{\lambda^2}\|x-\bar{y}\|^2 \ \ \ \ \text{for all} \ \ \ x\in X.
	$$
	Let $\bar{\varphi}(x)= -(a+\frac{\varepsilon}{\lambda^2})\|x\|^2 +\left\langle \ell+ \frac{2\varepsilon}{\lambda^2}\bar{y},x \right\rangle +c - \frac{\varepsilon}{\lambda^2}\|\bar{y}\|^2  - \varphi(\bar{y})+f(\bar{y})$. We have
	$$
	f(x)-f(\bar{y})\geq 	\begin{array}[t]{l}
	  -a\|x\|^2 +\left\langle \ell,x\right\rangle +a\|\bar{y}\|^2 -\left\langle \ell,\bar{y}\right\rangle - \frac{\varepsilon}{\lambda^2}\|x-\bar{y}\|^2 \\
	  \geq
	-a\|x\|^2 +\left\langle \ell,x\right\rangle +a\|\bar{y}\|^2 -\left\langle \ell,\bar{y}\right\rangle -\frac{\varepsilon}{\lambda^2}\|x\|^2 + \frac{2\varepsilon}{\lambda^2}\left\langle x, \bar{y}\right\rangle -\frac{\varepsilon}{\lambda^2}\|\bar{y}\|^2 \\
	\geq
	\left\langle \ell+\frac{2\varepsilon}{\lambda^2}\bar{y} ,x\right\rangle -
	(a+\frac{\varepsilon}{\lambda^2})\|x\|^2- \left\langle \ell+\frac{2\varepsilon}{\lambda^2}\bar{y} ,\bar{y}\right\rangle +	(a+\frac{\varepsilon}{\lambda^2})\|\bar{y}\|^2\\
 = \bar{\varphi}(x) -\bar{\varphi}(\bar{y}).
	  
	\end{array}
	$$
	Thus $\bar{\varphi}\in \partial_{\Phi_{lsc}}f(\bar{y})$. Moreover, 
	$$
	\|\ell - \bar{\ell} \|= \|-\frac{2\varepsilon}{\lambda^2}\bar{y} \|=\frac{2\varepsilon}{\lambda^2}\|\bar{y}\|\leq \frac{2\varepsilon}{\lambda^2} (\lambda + \|y\|), 
	$$
	$$
		c-\bar{c}=c-c + \frac{\varepsilon}{\lambda^2}\|\bar{y}\|^2  + \varphi(\bar{y})-f(\bar{y})\leq  \frac{\varepsilon}{\lambda^2}\|\bar{y}\|
	\ \ \text{and}\ \  \  
	\bar{a}-a=\frac{\varepsilon}{\lambda^2}.
	$$
\end{proof}

The domain of $\Phi_{lsc}$-subdifferential is defined as follows
$$
\textnormal{dom}(\partial_{\Phi_{lsc}}f):=\{x\in X \ :\ \partial_{\Phi_{lsc}}f(x)\neq \emptyset \}.
$$
The following theorem is an immediate corollary from Theorem 	\ref{i-r} and Proposition \ref{prop_subdiff}.
\begin{theorem}
	Let $X$ be a Hilbert space and  $f:X\rightarrow \bar{\mathbb{R}}$  be a proper $\Phi_{lsc}$-convex function. Then the set $\textnormal{dom}(\partial_{\Phi_{lsc}}f)$ is dense in $\textnormal{dom}(f)$.
\end{theorem}
The following technical fact  is used below.
	\begin{proposition}
		\label{bound}
		Let $X$ be a Hilbert space, $Z$ be a subset of $X$, $\alpha\in\mathbb{R}$ and  $\varphi_{1},\varphi_{2}:X\rightarrow \mathbb{R}$. If $\varphi_{1},\varphi_{2}$ have the  intersection property on $X$ at the level $\alpha$, then $\varphi_{1},\varphi_{2}$ have the  intersection property on $Z$ at the level $\alpha$.
	\end{proposition}

Theorem \ref{i-r} together with Theorem \ref{th-epsilon} allow to prove the 
main  result of this section.

\begin{theorem}
	\label{th-subdif}
	 Let $X$ be a Hilbert space and let $Z$ be a bounded subset of $X$. Let  $f,g:X\rightarrow \bar{\mathbb{R}}$ be proper 
	$\Phi_{lsc}$-convex functions and $\alpha\in\mathbb{R}$.
	
	Assume that there exist $\varphi_1\in\textnormal{supp}(f)$ and $\varphi_2\in\textnormal{supp}(g)$ such that $\varphi_1,\varphi_2$ have the  intersection property on $X$ at the level $\alpha$.
	
	Then  for any $\eta>0$ there exist  $x_{1} \in \textnormal{dom}(f)$, $x_{2}\in \textnormal{dom}(g)$ and $\bar{\varphi_1}\in\partial_{\Phi_{lsc}}f(x_{1})$ and $\bar{\varphi_2}\in\partial_{\Phi_{lsc}}g(x_{2})$ such that $\bar{\varphi_1},\bar{\varphi_{2}}$ have the  intersection property  on $Z$ at the level $\alpha-\eta$.
\end{theorem}

\begin{proof}
	By the boundedness of  $Z$, there exists $\gamma>0$ such that $Z \subset \gamma B(0,1)$, where $B(0,1)$ is the closed unit ball. Let  $\varphi_1,\varphi_2$ satisfy the assumptions of the theorem. Let $\eta>0$ and $\varepsilon=\frac{\eta}{\gamma}$.
	
	By Theorem \ref{th-epsilon}, there exist  $x_{1} \in \textnormal{dom}(f)$, $x_{2}\in \textnormal{dom}(g)$ and $\tilde{\varphi_1}(x)=-\tilde{a_{1}}\|x\|^2+\left\langle\tilde{ \ell_{1}},x\right\rangle+ \tilde{c_{1}} \in\partial_{\Phi_{lsc}}^{\varepsilon}f(x_{1})$ and $\tilde{\varphi_2}(x)=-\tilde{a_{2}}\|x\|^2+\left\langle\tilde{ \ell_{2}},x\right\rangle+ \tilde{c_{2}}\in\partial_{\Phi_{lsc}}^{\varepsilon}g(x_{2})$ such that $\tilde{\varphi_1},\tilde{\varphi_2}$ have the intersection property on $X$ at the level $\alpha$.
	
	Let  $\lambda_{1}=1+\sqrt{1+2\|x_{1}\|+\gamma +\frac{1}{\gamma}\|x_{1}\|^2}$ and  $\lambda_{2}=1+\sqrt{1+2\|x_{2}\|+\gamma +\frac{1}{\gamma}\|x_{2}\|^2}$. By Theorem  \ref{i-r}, there exist  $\bar{x_{1}} \in \textnormal{dom}(f)$, $\bar{x_{2}}\in \textnormal{dom}(g)$	 and $ \bar{\varphi_{1}}(x)=-\bar{a_{1}}\|x\|^2+\left\langle\bar{ \ell_{1}},x\right\rangle+ \bar{c_{1}}\in\partial_{\Phi_{lsc}}f(\bar{x}_{1})$
	and $\bar{\varphi_{2}}(x)=-\bar{a_{2}}\|x\|^2+\left\langle\bar{ \ell_{2}},x\right\rangle+\bar{c_{2}}\in\partial_{\Phi_{lsc}}g(\bar{x}_{2})$ such that
	\begin{equation}
	\label{eq-diff1}
	\|\tilde{\ell_{1}}- \bar{\ell_{1}} \|\leq \frac{2\varepsilon}{\lambda_{1}^2}(\lambda_{1} + \|x_{1}\|), \ \ \ \bar{a}_{1}-\tilde{a_{1}}=\frac{\varepsilon}{\lambda_{1}^2} \ \ \text{and} \
	\  \tilde{c_{1}}-\bar{c}_{1}\leq \frac{\varepsilon}{\lambda_{1}^2}\|x_{1}\|^2
	\end{equation}
	\begin{equation}
	\label{eq-diff2}
	\|\tilde{\ell_{2}}- \bar{\ell_{2}} \|\leq \frac{2\varepsilon}{\lambda_{2}^2}(\lambda_{2} + \|x_{2}\|), \ \ \ \bar{a}_{2}-\tilde{a_{2}}=\frac{\varepsilon}{\lambda_{2}^2} \ \ \text{and} \
	\  \tilde{c_{2}}-\bar{c}_{2}\leq \frac{\varepsilon}{\lambda_{2}^2}\|x_{2}\|^2.
	\end{equation}
 By \eqref{eq-diff1} and	\eqref{eq-diff2} and  by the boundedness of $Z$, for every $x\in Z$  we get
	\begin{equation}
	\label{eq-diff1-1}
	\tilde{\ell_{1}}(x) \leq  \bar{\ell_{1}}(x)+\frac{2\gamma\varepsilon}{\lambda_{1}^2}(\lambda_{1} + \|x_{1}\|), \  -\tilde{a_{1}}\|x\|^2=-\bar{a}_{1}\|x\|^2+\frac{\varepsilon}{\lambda_{1}^2}\|x\|^2\leq-\bar{a}_{1}\|x\|^2+\frac{\varepsilon}{\lambda_{1}^2}\gamma^2 ,
	\end{equation}
	\begin{equation}
	\label{eq-diff2-1}
	\tilde{\ell_{2}}(x) \leq \bar{\ell_{2}}(x)+\frac{2\gamma\varepsilon}{\lambda_{2}^2}(\lambda_{2} + \|x_{2}\|), \ -\tilde{a_{2}}\|x\|^2=-\bar{a}_{2}\|x\|^2+\frac{\varepsilon}{\lambda_{2}^2}\|x\|^2\leq -\bar{a}_{2}\|x\|^2+\frac{\varepsilon}{\lambda_{1}^2}\gamma^2 .
	\end{equation}
	By \eqref{eq-diff1}, \eqref{eq-diff2}, \eqref{eq-diff1-1}, \eqref{eq-diff2-1} and the definitions of $\varepsilon$, $\lambda_{1}$ and $\lambda_{2}$ we get
	\begin{equation}
	\label{eq-diff3}
	\tilde{\varphi_{1}}(x) \leq  \bar{\varphi_{1}}(x) +\eta,\ \ \tilde{\varphi_{2}}(x) \leq \bar{\varphi_{2}}(x)+\eta \ \ \text{for every}  \ \ x\in Z.
	\end{equation}
	Since the intersection property holds for $\tilde{\varphi_{1}}, \tilde{\varphi_{2}}$ on $X$ at the level $\alpha$, then by Proposition \ref{bound}, the intersection property holds for $\tilde{\varphi_{1}}, \tilde{\varphi_{2}}$ on $Z$ at the level $\alpha$.
	By Proposition \ref{prop-przeciecie}, the intersection property holds
	for $ \bar{	\varphi_{1}}+\eta$ and $\bar{\varphi_{2}}+\eta$ on $Z$ at the level $\alpha$, i.e. the intersection property holds for $ \bar{\varphi_{1}}$ and $\bar{\varphi_{2}}$ on $Z$ at the level $\alpha -\eta$. 
\end{proof}	

The example below shows that, in general, we cannot expect, the intersection property for $\Phi_{lsc}$-subgradients on the whole space $X$.

		\begin{example}
			
			Let $X=\mathbb{R}$,  $f(x):=2^{x}$, $g(x):=-|x|+2$ and $\alpha=0$.  
			The functions $f$ and $g$ are $\Phi_{lsc}$-convex. 
			Let $\varphi_{1}(x)\equiv 0$, then $\varphi_{1}\in \Phi_{lsc}$ and $\varphi_{1}\in\text{supp}(f)$. The set $[\varphi_{1}<0]$ is empty, then for every $\varphi_{2}\in \text{supp}(g)$ functions $\varphi_{1}$ and $\varphi_{2}$ have the intersection property on $X$ at the level $0$. 
			
			On the other hand,  for every $\bar{x}\in X$ and $\bar{\varphi}\in \partial_{\Phi_{lsc}}g(\bar{x})$ it must be
			$$
			\bar{\varphi}(x)=-ax^2+bx+c,
			$$
			where $a>0$, $b,c\in\mathbb{R}$. Since all $\Phi_{lsc}$-subgradients of  $f$ are affine functions,  there exist no $\Phi_{lsc}$-subgradient $\varphi_{f}$ of  $f$ and a $\Phi_{lsc}$-subgradient  $\varphi_{g}$ of $g$ such that  $\varphi_{f}$
			and  $\varphi_{g}$ have the intersection property on $X$ at the level $0$. 
			
			In view of  Theorem \ref{th-subdif}, for every bounded subset $Z\subset X$  there exist subgradients of  $f$ and subgradients of $g$ which have the intersection property on $Z$ at the level $0$.
		\end{example}
\section{The intersection property for $\Phi_{conv}$-subgradients}
Let $X$ be a topological vector space. 
The subdifferential $\partial f$ of function $f:X\rightarrow\bar{\mathbb{R}}$ at the point $\bar{x}\in \textnormal{dom}(f)$ is defined as follows
	$$
	\partial f(\bar{x}) := \{\ell \in  X^{*} \ : \ f(x)- f(\bar{x}) \geq \left\langle \ell,x -\bar{x}\right\rangle , \ \forall  x\in X\}.
	$$
There is an one-to-one correspondence between the $\Phi_{conv}$-subdifferential and the subdifferential of a convex lower semicontinuous function $f$, i.e. $\varphi=\ell+c\in \partial_{\Phi_{conv}} f(\bar{x})$ if and only if  $\ell \in \partial f(\bar{x}) $.

We recall that the normal cone to a convex set $C\subset X$ at the point $\bar{x}\in C$ is defined as
$$
N_{C}(\bar{x}):=\{\ell \in X^* \ : \ \left\langle \ell,\bar{x}-x \right\rangle\geq 0, \ \forall x\in C \},
$$
These definitions can be found e.g. in \cite{ek-tem}.

\begin{proposition}(\cite{ioffe-ti}, p. 200, Proposition 2)
	\label{norm}
	Let $X$ be a topological vector space. Let $\alpha\in\mathbb{R}$. Let $f:X\rightarrow\bar{\mathbb{R}}$ be proper convex and continuous and $[f<\alpha]\neq\emptyset$. Then
	$$
	N_{[f\leq\alpha]}(\bar{x})=\mathbb{R}_{+}\partial f(\bar{x}), \ \text{where} \ \ f(\bar{x})=\alpha.
	$$
\end{proposition}
Now we are ready to prove the main result of this section.
		\begin{theorem}
			\label{sub}
			Let $X$ be a topological vector space. Let $f,g:X\rightarrow \bar{\mathbb{R}}$ be proper continuous and convex functions  and $\alpha \in \mathbb{R}$. Assume that $f$ and $g$ have the intersection property on $X$ at the level $\alpha$
			and
			$$
			[f\leq\alpha]\cap [g\leq \alpha]\neq\emptyset.
			$$
			Then there exist $x_{1} \in \textnormal{dom}(f)$, $x_{2}\in \textnormal{dom}(g)$ and
			 $\varphi_{1} \in \partial_{\Phi_{conv}} f(x_{1})$, $\varphi_{2} \in \partial_{\Phi_{conv}} g(x_{2})$ such that $\varphi_{1},\varphi_{2}$  have the  intersection property on $X$ at the level $\alpha$.
		\end{theorem}
		\begin{proof}
			 Assume, first that $[f<\alpha]=\emptyset$, i.e. 
				$$
				f(x)\geq \alpha \ \ \ \ \text{for all} \ \ \ \ x\in X.
				$$
				By the assumption that $[f\leq\alpha]\cap [g\leq \alpha]\neq\emptyset$, there exists $x_{1}\in \textnormal{dom}(f)$ such that $f(x_{1})=\alpha$, and thus  $\bar{\varphi}\equiv \alpha$  belongs to $\partial_{\Phi_{conv}} f(x_{1})$. Let $x_{2}\in \textnormal{dom}(\partial_{\Phi_{conv}}g) $. Then  
				$$
				[\bar{\varphi}<\alpha]\cap [\varphi<\alpha]=\emptyset \ \ \ \text{for all} \ \ \ \varphi \in \partial_{\Phi_{conv}} g(x_{2}),
				$$
					which gives the assertion of the theorem.
					Analogously, by assuming that $[g<\alpha]=\emptyset$ we  get the assertion of the theorem.
			
			Assume now that $[f<\alpha]\neq \emptyset$ and $[g<\alpha]\neq\emptyset$.
			By the  intersection property of $f$ and $g$ on $X$ at the level $\alpha$, we have $[f<\alpha]\cap [g<\alpha]=\emptyset$.
			
		 We show that there exists $\ell\in X^{*}\setminus \{0\}$ such that
				\begin{equation}
				\label{phi}
				\sup_{y\in[f\leq\alpha]}\left\langle \ell,y\right\rangle= \left\langle \ell,\bar{x}\right\rangle= \inf_{z\in [g\leq\alpha]}\left\langle \ell,z\right\rangle,
				\end{equation}
				where $\bar{x}\in	[f\leq\alpha]\cap [g\leq \alpha] $.
				
				By the convexity and continuity of functions $f$, $g$, and by the fact that $[f<\alpha]\neq \emptyset$, $ [g< \alpha]\neq\emptyset$ we get $[f<\alpha]\cap [g\leq \alpha]=\emptyset$, $\text{int}[f\leq\alpha]=[f<\alpha]$ and $f(\bar{x})=g(\bar{x})=\alpha$. By the separation theorem (\cite{schirotzek},Theorem 1.5.3), there exist $\ell\in X^{*}\setminus \{0\}$ and $r\in\mathbb{R}$ such that
				$$
				\left\langle \ell,y\right\rangle \leq r\leq \left\langle \ell,z\right\rangle \ \ \ \text{for all} \ \ \  y\in[f\leq\alpha], \ z\in [g\leq\alpha].
				$$
				Let $\bar{x}\in	[f\leq\alpha]\cap [g\leq \alpha] $. 
				On the one hand
				$$
					\left\langle \ell,y\right\rangle\leq r\leq 	\left\langle \ell,\bar{x}\right\rangle \  \ \forall\ y\in[f\leq\alpha],
				$$
				hence $\sup\limits_{y\in[f\leq\alpha]}	\left\langle \ell,y\right\rangle\leq r\leq\left\langle \ell,\bar{x}\right\rangle.$
				On the other hand
				$$
				\left\langle \ell,\bar{x}\right\rangle\leq r\leq \left\langle \ell,z\right\rangle\ \ \forall z\in	[g\leq\alpha],
				$$
				thus $
				\left\langle \ell,\bar{x}\right\rangle\leq r\leq \inf\limits_{z\in [g\leq\alpha]}\left\langle \ell,z\right\rangle.$
				Then
				$$
				\left\langle \ell,\bar{x}\right\rangle\leq 	\sup_{y\in[f\leq\alpha]}\left\langle \ell,y\right\rangle\leq r\leq \left\langle \ell,\bar{x}\right\rangle\leq r\leq \inf_{z\in [g\leq\alpha]}\left\langle \ell,z\right\rangle\leq \left\langle \ell,\bar{x}\right\rangle,
				$$
				which proves \eqref{phi}, and consequently
			
			$$
			\sup_{y\in[f\leq\alpha]}\left\langle \ell,y\right\rangle=\left\langle \ell,\bar{x}\right\rangle.
			$$
			Hence
			$$
		\left\langle \ell,\bar{x}\right\rangle\geq \left\langle \ell,y\right\rangle \ \ \text{for all } \ \ y\in[f\leq\alpha],
			$$
		or equivalently
			$$
			\left\langle \ell,\bar{x}-y\right\rangle\geq 0 \ \ \text{for all } \ \ y\in[f\leq\alpha].
			$$
			By Proposition \ref{norm}
			$$
			\ell \in N_{[f\leq\alpha]}(\bar{x})=\mathbb{R}_{+}\partial f(\bar{x}).
			$$
			Hence, there exists $k>0$ such that $\ell \in \partial k f(\bar{x})$, i.e.
			$$
			f(x)\geq \frac{1}{k} \left\langle \ell,x-\bar{x}\right\rangle+f(\bar{x}), \ \ \text{for all} \ \ x\in X.
			$$
			Let $\varphi_{1}(x):= \frac{1}{k} \left\langle \ell,x-\bar{x}\right\rangle+f(\bar{x})$. Then $\varphi_{1}\in\partial_{\Phi_{conv}} f(\bar{x}) $ and
			$\varphi_{1}\in \text{supp}(f)$ and
			$$
			[\varphi_{1}<\alpha]=\{x\in X\ : \ \left\langle \ell,x-\bar{x}\right\rangle<0\}.
			$$
			By \eqref{phi},
			$$
			\left\langle \ell,\bar{x}\right\rangle=\inf_{z\in [g\leq\alpha]}\left\langle \ell,z\right\rangle,
			$$
			which is equivalent to 
			$$
			-\left\langle \ell,\bar{x}\right\rangle=\sup_{z\in [g\leq\alpha]}-\left\langle \ell,z\right\rangle,
			$$
			i.e.
			$$
			-\left\langle \ell,\bar{x}\right\rangle\geq -\left\langle \ell,z\right\rangle \ \ \text{for all } \ \ z\in[g\leq\alpha],
			$$
		or
			$$
			-\left\langle \ell,\bar{x}-z\right\rangle\geq 0 \ \ \text{for all } \ \ z\in[g\leq\alpha].
			$$
			Again, by Proposition \ref{norm}
			$$
			-\ell \in N_{[g\leq\alpha]}(\bar{x})=\mathbb{R}_{+}\partial g(\bar{x}).
			$$
			There exists $\lambda>0$ such that $-\ell \in \partial \lambda g(\bar{x})$. We have
			$$
			g(x)\geq -\frac{1}{\lambda} \left\langle \ell,x-\bar{x}\right\rangle+g(\bar{x}), \ \ \text{for all} \ \ x\in X.
			$$
			Let $\varphi_{2}(x):=  -\frac{1}{\lambda} \left\langle \ell,x-\bar{x}\right\rangle+g(\bar{x})$. Then $\varphi_{2}\in\partial_{\Phi_{conv}} g(\bar{x}) $ and
			$\varphi_{2}\in \text{supp}(g)$ and
			$$
			[\varphi_{2}<\alpha]=\{x\in X \ : \ \left\langle \ell,x-\bar{x}\right\rangle >0\}.
			$$
			Thus
			$$
			[\varphi_{1}<\alpha]\cap [\varphi_{2}<\alpha]=\emptyset.
			$$
		\end{proof}	
		
		The following simple fact was proved in \cite{syga}.
		\begin{proposition}(\cite{syga}, Proposition 2.1)
			\label{leq}
			If $\varphi_{1},\varphi_{2}$ have the intersection property on $X$ at the level $\alpha\in \mathbb{R}$,
			then $\varphi_{1},\varphi_{2}$ have the intersection property on $X$ at 
			any level $\beta \leq \alpha$, $\beta\in \mathbb{R}$.
		\end{proposition}
		Taking into account that above proposition we can formulate the following corollary of Theorem \ref{sub}.
\begin{corollary}
		Let $f,g:X\rightarrow \bar{\mathbb{R}}$ be proper continuous and convex  and $\alpha \in \mathbb{R}$. If $f$ and $g$ have the intersection property on $X$ at the level $\alpha$
		and
		$$
		[f\leq\alpha]\cap [g\leq \alpha]\neq\emptyset
		$$
			then there exists $x_{1} \in \textnormal{dom}(f)$, $x_{2}\in \textnormal{dom}(g)$ and
			$\varphi_{1} \in \partial_{\Phi_{conv}} f(x_{1})$, $\varphi_{2} \in \partial_{\Phi_{conv}} g(x_{2})$ such that $\varphi_{1},\varphi_{2}$  have the intersection property on $X$ at every level $\beta\leq \alpha$.
	\end{corollary}

	\section{Main results}
	
	In this section we prove minimax theorems for functions
	 $a:X\times Y\rightarrow \bar{\mathbb{R}}$ which are $\Phi_{lsc}$-($\Phi_{conv}$-)convex
	 with respect to variable $x$ for every $y\in Y$. In these theorems
	  we do not use any compactness and/or connectedness assumptions.
	  
	  For $\Phi_{lsc}$-functions we prove the following sufficient conditions for the minimax equality to hold.
	  
\begin{theorem}
			\label{min-max-lsc-n}
			Let $X$ be a  Hilbert space and $Y$ be a real vector space. Let $a:X\times Y\rightarrow\bar{\mathbb{R}}$. Assume that for any $y\in Y$ the  function $a(\cdot,y):X\rightarrow\bar{\mathbb{R}}$ 
			is proper $\Phi_{lsc}$-convex on $X$ and for any $x\in X$ the function $a(x,\cdot):Y\rightarrow\bar{\mathbb{R}}$  is concave on $Y$. 
			
			If for every $\alpha\in\mathbb{R}$, $\alpha < \inf\limits_{x\in X} \sup\limits_{y\in Y} a(x,y)$ there exist $y_{1}, y_{2}\in Y$,  $x_{1} \in \textnormal{dom}(a(\cdot, y_{1}))$, $x_{2}\in \textnormal{dom}(a(\cdot, y_{2}))$ and $\varphi_{1}\in \partial_{\Phi_{lsc}} a(\cdot, y_{1})(x_{1})$, $\varphi_{2}\in \partial_{\Phi_{lsc}}a(\cdot, y_{2})(x_{2})$ such that the intersection property holds for $\varphi_{1}$ and $\varphi_{2}$ on $X$ at the level $\alpha$, 
			then 
			$$\sup\limits_{y\in Y} \inf\limits_{x\in X} a(x,y)=\inf\limits_{x\in X} \sup\limits_{y\in Y} a(x,y).$$
			
		\end{theorem}
		\begin{proof}
			 Follows immediately from Proposition \ref{prop_subdiff} and
			Theorem \ref{min-max}.
			\end{proof}
			
			\begin{remark}
				Let us observe that if for some $y_{1}, y_{2}\in Y$
				and  $x_{1} \in \textnormal{dom}(a(\cdot, y_{1}))$, $x_{2}\in \textnormal{dom}(a(\cdot, y_{2}))$ and $\varphi_{1}(\cdot)=-a_{1}\|\cdot\|^2+\left\langle \ell_{1},\cdot\right\rangle+c_{1}\in \partial_{\Phi_{lsc}} a(\cdot, y_{1})(x_{1})$, $\varphi_{2}(\cdot)=-a_{2}\|\cdot\|^2+\left\langle \ell_{2},\cdot\right\rangle+c_{2}\in \partial_{\Phi_{lsc}}a(\cdot, y_{2})(x_{2})$ the intersection property holds for $\varphi_{1}$ and $\varphi_{2}$ on $X$ at the level $\alpha$, 
				then exactly one of the conditions hold
				\begin{description}
				 \item [{\em (i)}] $\varphi_{1}:=const\ge\alpha$ or    $\varphi_{2}:=const\ge\alpha$,
				 \item [{\em (ii)}] $\varphi_{1}:=\langle\ell_{1},\cdot\rangle+c_{1}$ and
					$\varphi_{2}:=\langle\ell_{2},\cdot\rangle+c_{2}$, i.e. $a_{1}=a_{2}=0$.
				\end{description}
				\end{remark}
			To show the necessity of the intersection property  of Theorem \ref{min-max-lsc-n}  we restrict our attention to bounded subsets of the Hilbert space $X$.
			
			\begin{theorem}
				\label{min-max-lsc-m2}
				Let $X$ be a  Hilbert space, $Y$ be a real vector space and let $a:X\times Y\rightarrow\bar{\mathbb{R}}$. Assume  that for any $y\in Y$ the  function $a(\cdot,y):X\rightarrow\bar{\mathbb{R}}$ 
				is proper $\Phi_{lsc}$-convex on $X$. 
				
				If $\sup\limits_{y\in Y} \inf\limits_{x\in X} a(x,y)=\inf\limits_{x\in X} \sup\limits_{y\in Y} a(x,y)$, then  for every $\alpha\in\mathbb{R}$, $\alpha < \inf\limits_{x\in X} \sup\limits_{y\in Y} a(x,y)$, and every bounded set $Z\subset X$ 
				 there exist $y_{1}, y_{2}\in Y$,  $x_{1} \in \textnormal{dom}(a(\cdot, y_{1}))$, $x_{2}\in \textnormal{dom}(a(\cdot, y_{2}))$ and $\varphi_{1}\in \partial_{\Phi_{lsc}} a(\cdot, y_{1})(x_{1})$, $\varphi_{2}\in \partial_{\Phi_{lsc}}a(\cdot, y_{2})(x_{2})$ such that the intersection property holds for $\varphi_{1}$ and $\varphi_{2}$ on $Z$ at the level $\alpha$,
			\end{theorem}
	
			\begin{proof}
		 Take any $	\alpha < \inf\limits_{x\in X} \sup\limits_{y\in Y} a(x,y)$.
			By the minimax equality,
			$\sup\limits_{y\in Y} \inf\limits_{x\in X} a(x,y)>\alpha$, i.e. there exists $\bar{y}\in Y$ such that for every $x\in X$ we have
			$$
			a(x,\bar{y})> \alpha.
			$$
			The function
			$\bar{\varphi}:= \alpha$ belongs to  $\text{supp} \ a(\cdot,\bar{y})$. We have  $[\bar{\varphi}<\alpha]=\emptyset$,
			thus
			$$
			[\varphi<\alpha]\cap [\bar{\varphi}<\alpha]=\emptyset \ \ \text{for all} \ \  \varphi \in \Phi_{lsc},
			$$
			i.e. for every $\varphi \in \Phi_{lsc}$ the functions $\bar{\varphi}$ and $\varphi$ have the intersection property on $X$ at the level $\alpha$.
			
			Take any $\inf\limits_{x\in X} \sup\limits_{y\in Y} a(x,y)> \beta>\alpha$ and set $\eta:=\beta-\alpha>0$. By the same arguments as above there exist $y_{1},y_{2}\in Y$, $\varphi_{1}\in \text{supp}\ a(\cdot,y_{1})$ and  $\varphi_{2}\in \text{supp}\ a(\cdot,y_{2})$ such that the  intersection property holds for  functions $\varphi_{1}$ and $\varphi_{2}$ on $X$ at the level $\beta$. Let $Z$ be a bounded subset of $X$. By Theorem \ref{th-subdif}, there exist $x_{1} \in \textnormal{dom}(a(\cdot, y_{1}))$, $x_{2}\in \textnormal{dom}(a(\cdot, y_{2}))$  and $\bar{\varphi_{1}}\in \partial_{\Phi_{lsc}} a(\cdot, y_{1})(x_{1})$, $\bar{\varphi_{2}}\in \partial_{\Phi_{lsc}}a(\cdot, y_{2})(x_{2})$ such that the intersection property holds for $\bar{\varphi_{1}}$ and $\bar{\varphi_{2}}$ on $Z$ at the level $\beta-\eta$, then the
			intersection property holds for $\bar{\varphi_{1}}$ and $\bar{\varphi_{2}}$ on $Z$ at the level $\beta-(\beta-\alpha)=\alpha$.  
		\end{proof}
		
	In the class of $\Phi_{conv}$-convex functions we get the following sufficient and 
	necessary conditions for the minimax equality.

	\begin{theorem}
		\label{min-max-conv_n}
		Let $X$ be a reflexive Banach space and $Y$ be a real vector space. Let $a:X\times Y\rightarrow\bar{\mathbb{R}}$ be a function such that for any $y\in Y$ the  function $a(\cdot,y):X\rightarrow\mathbb{R}$ 
		is proper continuous and convex on $X$ and for any $x\in X$ the function $a(x,\cdot):Y\rightarrow\mathbb{R}$  is concave on $Y$. Assume that
		there exists a point $\bar{y}\in Y$ such that $\sup\limits_{y\in Y} \inf\limits_{x\in X} a(x,y)=\inf\limits_{x\in X} a(x,\bar{y})$ and there exists $\tilde{y}\in Y$ and $\gamma>\sup\limits_{y\in Y} \inf\limits_{x\in X} a(x,y)$ such that the set $[a(\cdot, \tilde{y})\leq \gamma]$ is bounded in $X$.
	 The following conditions are equivalent:
		\begin{description}
			\item [{\em (i)}] for every $\alpha\in\mathbb{R}$, $\alpha < \inf\limits_{x\in X} \sup\limits_{y\in Y} a(x,y)$ there exist $y_{1}, y_{2}\in Y$, $x_{1} \in \textnormal{dom}(a(\cdot, y_{1}))$, $x_{2}\in \textnormal{dom}(a(\cdot, y_{2}))$ and $\varphi_{1}\in \partial_{\Phi_{conv}} a(\cdot, y_{1})(x_{1})$, $\varphi_{2}\in \partial_{\Phi_{conv}}a(\cdot, y_{2})(x_{2})$ such that the intersection property holds for $\varphi_{1}$ and $\varphi_{2}$ on $X$ at the level $\alpha$,
			\item [{\em (ii)}] $\sup\limits_{y\in Y} \inf\limits_{x\in X} a(x,y)=\inf\limits_{x\in X} \sup\limits_{y\in Y} a(x,y).$
		\end{description}
	\end{theorem}
	\begin{proof}
		$(i)\Rightarrow (ii)$ follows immediately from Proposition \ref{prop_subdiff} and
		Theorem \ref{min-max}.
		
			$(ii)\Rightarrow (i)$ Let $\beta:=
			\inf\limits_{x\in X} \sup\limits_{y\in Y} a(x,y)= \sup\limits_{y\in Y} \inf\limits_{x\in X} a(x,y). $
			
			By assumption,  
			$$
			a(x,\bar{y})\geq \beta \ \  \ \text{for every} \ \  x\in X.
			$$
			Hence, for every $\hat{y}\in Y$  the functions $a(\cdot,\bar{y}),a(\cdot,\hat{y}):X\rightarrow\bar{\mathbb{R}}$ have the  intersection property on $X$ at the level $\beta$. Moreover, for every $\varepsilon>0$ there exists $x_{\varepsilon}\in X$ such that 
		 	 \begin{equation}
		 	 \label{inf}
		 	a(x_{\varepsilon},y)< \beta+\varepsilon \ \ \ \text{for every} \ \ \ y\in Y. 
		 	\end{equation}
		 	Let $\tilde{y}\in Y$ and $\tilde{\varepsilon}$ be such that the set $[a(\cdot, \tilde{y})\leq \beta+\tilde{\varepsilon }]$ is bounded.
		 	Consider the sets
		 	$$
		 	A_{n}:=[a(\cdot, \tilde{y})\leq \beta+\frac{1}{n}]\cap [a(\cdot, \bar{y})\leq \beta+\frac{1}{n}], \ \ \ \text{for all} \ \ n\in\mathbb{N}.
		 	$$
		 	The sets $A_{n}$ are nonempty (by \eqref{inf}), $A_{n+1}\subset A_{n}$, convex, closed and bounded for all $n\geq n_{0}$. By the Cantor's intersection theorem, there exists $\bar{x}\in X$ such that
		 	$$
		 	\bar{x}\in \bigcap_{n\geq n_{0}}^{\infty}A_{n},
		 	$$
		 	thus, $	\bar{x}\in[a(\cdot, \tilde{y})\leq \beta]\cap [a(\cdot, \bar{y})\leq \beta]$.
		 	Hence, by Theorem 	\ref{sub}, there exist $x_{1} \in \textnormal{dom}(a(\cdot, \bar{y}))$, $x_{2}\in \textnormal{dom}(a(\cdot, \tilde{y}))$ and $\varphi_{1}\in \partial_{\Phi_{conv}} a(\cdot,\bar{y} )(x_{1})$, $\varphi_{2}\in \partial_{\Phi_{conv}}a(\cdot,\tilde{y} )(x_{2})$ such that the intersection property holds for $\varphi_{1}$ and $\varphi_{2}$ on $X$ at the level $\beta$. 
		 	
		 	Let $\alpha < \inf\limits_{x\in X} \sup\limits_{y\in Y} a(x,y)=\beta$. By Proposition \ref{leq}, functions $\varphi_{1}$ and $\varphi_{2}$ have the intersection property on $X$ at the level $\alpha$.
	\end{proof}

	 Taking into account Theorem \ref{th-epsilon} we can formulate the following sufficient and 
	 necessary conditions for the minimax equality in terms of $\Phi$-$\varepsilon$-subgradients  for  $\Phi$-convex functions.
		\begin{theorem}
			\label{min-max-ep}
			Let $X$ be a set and $Y$ be a vector space. Let $a:X\times Y\rightarrow\bar{\mathbb{R}}$ be a function such that for any $y\in Y$ the  function $a(\cdot,y):X\rightarrow\mathbb{R}$ 
			is proper $\Phi$-convex on $X$ and for any $x\in X$ the function $a(x,\cdot):Y\rightarrow\mathbb{R}$  is concave on $Y$.  The following conditions are equivalent
			\begin{description}
				\item [{\em (i)}] for every $\alpha\in\mathbb{R}$, $\alpha < \inf\limits_{x\in X} \sup\limits_{y\in Y} a(x,y)$ and for every $\varepsilon>0$ there exist $y_{1}, y_{2}\in Y$, $x_{1} \in \textnormal{dom}(a(\cdot, y_{1}))$, $x_{2}\in \textnormal{dom}(a(\cdot, y_{2}))$ and $\varphi_{1}\in \partial_{\Phi}^{\varepsilon} a(\cdot, y_{1})(x_{1})$, $\varphi_{2}\in \partial_{\Phi}^{\varepsilon}a(\cdot, y_{2})(x_{2})$ such that the intersection property holds for $\varphi_{1}$ and $\varphi_{2}$ on $X$ at the level $\alpha$,
				\item [{\em (ii)}] $\sup\limits_{y\in Y} \inf\limits_{x\in X} a(x,y)=\inf\limits_{x\in X} \sup\limits_{y\in Y} a(x,y).$
			\end{description}
		\end{theorem}
		\begin{proof}
			Follows immediately from  Theorem \ref{th-epsilon} and Theorem \ref{min-max}.
		\end{proof}

\end{document}